\documentclass{article}
\usepackage{amsthm,amsfonts,xfrac}
\usepackage{authblk}
\usepackage[backend=biber,style=alphabetic,sorting=ynt]{biblatex}
\usepackage[utf8]{inputenc}
\usepackage{csquotes}
\usepackage[english]{babel}
\usepackage{mathtools} 
\usepackage{pgfplots}
\usepackage[font=small,labelfont=bf]{caption} 
\DeclarePairedDelimiter\floor{\lfloor}{\rfloor}
\pgfplotsset{compat=1.12}
\newtheorem{theorem}{Theorem}
\newtheorem*{thm}{Theorem}
\theoremstyle{definition}
\newtheorem{definition}{Definition}[section]
\theoremstyle{definition}
\newtheorem{lemma}{Lemma}[section]
\theoremstyle{definition}
\newtheorem{exmp}{Example}[section]
\newtheorem{corollary}{Corollary}[theorem]

\newtheorem{conjecture}{Conjecture}[theorem]
\newcommand{\R}{\mathbb{R}}

\providecommand{\keywords}[1]
{
  \small	
  \textbf{\textit{Keywords---}} #1
}

\providecommand{\msc}[1]
{
  \small	
  \textbf{\textit{Mathematics Subject Classification---}} #1
}

\title{Collapsing Maps and Quasi-Isometries}

\author{Joshua Thompson}
\author{Davin Hemmila}
\affil{Department of Mathematics and Computer Science, Northern Michigan University}

\date{\today}

\begin{document}

\maketitle

\begin{abstract}
We introduce a generalization of the b-metric we call a (b,c)-metric.  We show that if $X$ is a $(b,c)$-metric space and $\psi: X \longrightarrow Y$ is a quasi-isometry then $Y$ is $(b,c)$-metrizable.  We also define a particular kind of collapsing map that can be applied to an arbitrary $(b,c)$-metric space.  We define a distance function on the image of this collapsing map and with this prove that the collapsing map is a quasi-isometry.  
\end{abstract}


\msc{[2020],51F30,54B15}

\section{Introduction}

An isometry is a distance-preserving transformation of a metric space and a quasi-isometry is a transformation that distorts distances by a bounded amount.  Quasi-isometries are found in many places in mathematics and are ubiquitous in low-dimensional topology and geometric group theory \cite{Clay}.  The set of all isometries on a given metric space is a subject of classical study, see for example \cite{Dan}.  The set of all quasi-isometries on a given metric space is undoubtedly much more complex but also much less studied. Many examples of quasi-isometries are similar to self-maps that collapse or glue together distinct subsets.  Collapsing maps are a subset of all possible transformations of this form.  We show that if a map collapses a certain space in a certain way it will always be a quasi-isometry.

What we will mean by a certain space is that it be a globally Lipschitz $(b,c)$-metric space.  The Lipschitz condition prevents the space from being stretched too far in the new metric.  The $(b,c)$-metric space is a simple generalization of $b$-metric space studied by \cite{Tom}, which are almost metric spaces except the triangle inequality is relaxed by a multiplicative constant.  A $(b,c)$-metric is almost a metric space except that its triangle inequality is relaxed by both an additive and multiplicative constant.  Those familiar with quasi-isometries will note the obvious connection to $(b,c)$-metric spaces.     
 
 What we will mean for a map to collapse in a certain way is that do so in a bounded fashion. This means two things.  First, our collapsing maps are self-maps that move points by at most a bounded distance. Second, the distortion is controlled by a Lipschitz constant.  These two features of the collapsing map serve to control the distortion in the mapping and give rise to a quasi-isometry.  Our main theorem is as follows.
 
 \begin{thm}
    Let $(X,\rho)$ be a $(b,c)-$metric space and let $\phi: X \longrightarrow \sfrac{X}{\sim}$ be a collapsing map that collapses subset $S$ via equivalence relation $\sim$. If $\rho_\phi$ is the collapsed metric on $\sfrac{X}{\sim}$ then there are constants $K$ and $C$ such that 
    \[\frac{1}{K}\rho(x,y) - C \leq \rho_\phi([x],[y]) \leq  K\rho(x,y)+C \mbox{ } \forall x,y \in X \]
    Then $\phi$ a quasi-isometry between $(X,\rho)$ and $(\sfrac{X}{\sim},\rho_\phi)$.
   
\end{thm}

 \subsection{Quasi-Isometry and $(b,c)$-metric Spaces}
 
 Here we define the basic terms and prove a Lemma which shows that quasi-isometries preserve $(b,c)$-metrics.  
 
 \begin{definition}
 A quasi-isometry is a mapping between a $(b,c)$-metric space X and a semi-metric Y, $f:X \to Y$ such that $$\frac{1}{K}\rho(x,y)-C \leq \rho_\phi(f(x),f(y)) \leq K \rho(x,y) + C$$
for some constants $C \in [0,\infty), K \in [1,\infty)$.
 \end{definition}
 
 \begin{exmp}\label{example_1} The step function is a quasi-isometry from the reals to the integers. Let $f: \mathbb{R}\longrightarrow \mathbb{Z}$ by $f(x) = \floor x$ be the least integer function.  Then $f$ is quasi-isometry from $\mathbb{R}$ to $\mathbb{Z}$ with $K=1$ and $C=1$.
 \end{exmp}
 
 We are particularly interested in maps that collapse subsets, similar to what happens under quotient maps of topological spaces.  It is well known that a bounded metric space is quasi-isometric to a point.  Therefore, the following example should not be surprising.

 \begin{exmp}
 Let $X = \mathbb{R}^2$ and let $D$ be the closed unit disk.  Let $f$ be the quotient map that collapses the unit disk to a point.  Then $f:X \longrightarrow Y$ where $Y = \sfrac{X}{\sim}$ where $x \sim y$ if both $x$ and $y$ lie in the closed unit disk.   One can define a metric $\sigma$ on $Y$ by requiring that $\sigma([x]) = \rho(x)$ where $\rho$ is standard metric.  It follows that $(Y,\sigma)$ is quasi-isometric to $(X,\rho)$.
 \end{exmp}
 
 In this paper we collapse unbounded subsets, but they are unbounded in only one dimension.  These subsets are collapsed to sub-manifolds that have bounded variation (actually Lipschitz continuous) so the unbounded nature of the collapsed subset is effectively controlled by the Lipschitz condition.  So while general collapsing maps can kill the triangle inequality entirely we prove that our so-called Lipschitz collapsings give rise to $b$-metrics, metric spaces with a relaxed triangle inequality.

 This puts in the category of $b$-metric spaces, a generalization of metric spaces introduced by Bourbaki, and recently considered by Tomonoari \cite{Tom} .
 
 \begin{definition}
   A $b$-metric space is a pair (X, d) where X is a set and $d:X \times X \to \mathbb{R}^{+0}$ satisfying for all $x,y \in X$
 \begin{enumerate}
    \item $d(x,y) \geq 0$ with $d(x,y) = 0 \iff x = y$
    \item $d(x,y) = d(y,x) $
    \item There exists $b\geq 1$ such that $d(x,z) \leq b(d(x,y) + d(y,z))$
 \end{enumerate}

 \end{definition}
 
 \begin{exmp} \label{collapse rn}We collapse the unit ball $D$ in $\mathbb{R}^n$.   Let $\rho$ be the standard metric on $\mathbb{R}^n$ and let $\sim$ be the relation $x \sim y$ if both $x\in D$ and $y \in D$.  If $g$ is the geodesic between $x$ and $y$ let $E = g \cap D$ and $u$ be the length of $E$.  It is immediate that $\rho(x,y) = u+v$ for some $v \in \mathbb{R}$.   If we define $\sigma$ by $\sigma([x],[y]) = \rho(x,y) - u$ it follows that $\sigma$ is a $K$-metric on $\sfrac{X}{\sim}$ for $K = 2$.  The triangle inequality does not hold for all points in the collapsed space.  For example in $\mathbb{R}^2$, if $a=(-1.1,0)$, $b=(1.1,0)$ and $c =(1,1, 10)$ $$\sigma([a],[b]) + \sigma([b],[c]) = .2 + 10 < 10.24 \approx \sigma([a],[c]), $$  however $\sigma([a],[c]) \leq 2(\sigma([a],[b]) + \sigma([b],[c]))$. It is also the case that this map is a quasi-isometry of $\mathbb{R}^n$ with K = 1 and C = 1. 
 \end{exmp}

The distortion of a metric caused by a quasi-isometry is of a bounded nature and we show that this bounded distortion merely relaxes the triangle inequality, instead of destroying it.   Spaces that satisfy a "relaxed triangle inequality" are $b-$metric spaces as defined by Czerwik (1998).  That is, after multiplying by a constant $b$ a triangle inequality holds.  We expand this notion to a $(b,c)$-metric, where we multiply by $b$ and add $c$.  This is the appropriate setting to describe quasi-isometries.  
 
 By the term $\textit{semi-metric space}$ we refer to a space with a non-degenerate symmetric bi-linear form; that is, a metric space without a triangle inequality.

 \begin{definition}
 A (b,c)-metric space is a semi-metric space with the extra condition that 
 \[ d(x,z) \leq b(d(x,y) + d(y,z)) + c\]
 where $b \geq 1$ and $c \geq 0$.
 \end{definition}
 
 (b,c)-metrics arise from applying a quasi-isometry to a metric space without the restriction that the range must have a certain triangle inequality. 
 
 \begin{exmp}
 The Euclidean plane is a $(1,0)$-metric space and Example \ref{collapse rn} is a $(1,1)$-metric space.
 \end{exmp}
 
 \begin{lemma}
     If $f: X \rightarrow Y$ is a quasi-isometry where $X$ is a $(b,c)$-metric space and $Y$ is a semi-metric space, then $Y$ is a $(b',c')$-metric space.
        \end{lemma}
 
 \begin{proof}
Since $f$ is a quasi-isometry and $X$ has a relaxed triangle inequality, we know: 
\[\rho_\phi(f(x),f(z)) \leq K\rho(x,z) + C \tag{1} \]
\[ \frac{1}{K} \rho(x,y) - C \leq \rho_\phi(f(x),f(y)) \tag{2} \] 
\[ \frac{1}{K} \rho(y,z) - C \leq \rho_\phi(f(y),f(z)) \tag{3} \] 
\[ \rho(x,z)\leq b(\rho(x,y) + \rho(y,z)) + c \tag{4}\]

Multiplying (4) by K and adding C we get:

\[ K\rho(x,z) + C\leq Kb(\rho(x,y) + \rho(y,z)) + Kc + C \tag{5}\]

Bringing the inequality in (1) to (5) we get:

\[\rho_\phi(f(x),f(z)) \leq K\rho(x,z) + C\leq Kb(\rho(x,y) + \rho(y,z)) + Kc + C \tag{6}\]

Multiplying (2) and (3) by bK$^2$, adding bCK$^2$ + $\frac{Kc + C}{2}$ and bringing the sum of (2) and (3) to (6) we get:

\[Kb(\rho(x,y) + \rho(y,z)) + Kc + C \leq\]
\[bK^2(\rho_\phi(f(x),f(y)) + \rho_\phi(f(y),f(z))) + 2bCK^2 + Kc + C\]

so
\[\rho_\phi(f(x),f(z)) \leq bK^2(\rho_\phi(f(x),f(y)) + \rho_\phi(f(y),f(z))) + 2bCK^2 + Kc + C\]

and setting bK$^2$ = $b'$ and 2bCK$^2$ + Kc + C = $c'$ we see Y is a $(b',c')$-metric.
 \end{proof}

\subsection{Collapsing Maps}
In this section we define a collapsing map and show that it preserves the $(b,c)$-metric structure of a space.

Given a (b,c)-metric space $X$ and a subset $S$ we wish to define a mapping on $X$ wherein we collapse $S$ to a proper subset of itself and glue the complement of $S$ in $X$ together along this proper subset.  This can be viewed as collapsing $S$ and \textit{bringing the complement along for the ride}.  This is quite similar to the concept of a quotient map of topological spaces.  Collapsing a set $S$ can be thought of partitioning $S$ into equivalence class and mapping every element of $S$ to its equivalence class.  

Quotient maps in metric spaces can be quite general and can destroy the metric completely.  Since we are collapsing a  To give ourselves some control we assume hereafter that the $(b,c)$-semi-metric space $(X,\rho)$ is a (globally) Lipschitz manifold.

\begin{definition}[Lipschitz] A function $f:X \longrightarrow Y$ is (globally) Lipschitz continuous if there exists a real constant $L \geq 0$ such that for all $x_1, x_2 \in x$ 
\[d_Y(f(x_1),f(x_2)) \leq L \cdot d_X(x_1,x_2) \]
We say the function $f$ is bi-Lipschitz if both $f$ and its inverse are Lipschitz. 
\end{definition}

\begin{exmp}
If $S = \{(x,y) \in \mathbb{R}^2 | y = \sin(x)+1$ then $S$ is the image of a (globally) Lipschitz function with Lipschitz constant $L = 1.$
\end{exmp}

Since the derivative of a Lipschitz function is bounded we can think of a (globally) Lipschitz manifold as one whose "steepness" is bounded above by a constant.  As we shall see, without this Lipschitz condition one can construct a collapsing of a smooth sub-manifolds that does not give quasi-isometry to our collapsed metric. There is a close relationship between quasi-conformal maps and Lipschitz homeomorphisms, see \cite{Luu}, but they do not consider quasi-isometries.

The sets we collapse come equipped with a fibering.  Recall, a fibered manifold is a manifold $E$ together with a projection $\pi: E \rightarrow c$ that is a surjective submersion.  Every point of $E$ is then contained in some fiber $\pi^{-1}(x)$.

\begin{definition}[Bounded fibered manifold] A fibered manifold where the length of every fiber is bounded.  
\end{definition}

\begin{exmp}
In the example above, the unit disk is a bounded fibrered manifold.  Note that the region of $\mathbb{R}^2$ between the graphs of $y=x^2$ and $y=x$ may be fibered using vertical lines, but this fibering is not bounded.  
\end{exmp}

For a given metric space $(X,\rho)$ and subset $T \subset X$, there are two natural metric spaces we consider: $(T, \rho)$ and $(T,\rho_p)$ where $\rho_p(a,b)$ is defined to be the infimum of all paths \textit{in T} from $a$ to $b$.  In our collapsing map $T$ can be thought of as the image of a retraction and we need this image to not be too wild. Our collapsing map definition requires the existence of a set which is to be collapsed.
 
\begin{definition}[Collapsing Set] \label{collapsing set} Let $X$ be an $n-$manifold with metric $\rho$.  The collapsing subset $S \subset X$ has the following properties.

\begin{enumerate}
    \item $S$ is a bounded fibered $n$-dimensional sub-manifold of $X$ with fibering $\pi: E \rightarrow X$.
    \item $T \subset S$ is a (n-1)-manifold transverse to each fiber $\pi^{-1}(x)$.
    \item The identity mapping $i: (T,\rho) \longrightarrow (T,\rho_p)$ is Lipschitz where $\rho_p$ is the path-metric in $T$ inherited from $\rho$.
\end{enumerate}

If a curve $T$ satisfies property (3) above we say that $T$ is a Lipschitz curve in $X$.
    \end{definition}
 
\begin{exmp}
Let $S$ be the region of $\mathbb{R}^2$ between $y = -2$ and $y = 2$.  This region can be fibered by vertical segments.  Let $T = \{(x,y) \in \mathbb{R}^2 \mbox{ } | \mbox{ } y = \cos(2x)\}$. It is clear that $T$ is transverse to all vertical segments in $S$ and is the image of a Lipschitz function $f(x) = \cos(2x)$ with Lipschitz constant $L = 2$. Consequently the derivative is bounded at all $x$: \[|f'(x)| \leq 2\] thus \[ \rho_p(a,b) = \int_a^b \sqrt{1 + (f'(t))^2}dt = \int_a^b \sqrt{1 + 4\sin^2(2t)} dt \]
\[ \leq \int_a^b \sqrt{1+4(2)^2} dt = \sqrt{17}*|b-a| = \sqrt{17}\rho(a,b)\] where $\rho$ is the standard metric on $\mathbb{R}^2$.  Thus the identity mapping $i:(T,\rho) \longrightarrow (T,\rho_p)$ is Lipschitz.
\end{exmp}

Let $S \subset X$ be a collapsing set. We define a relation $\sim$ called the \textit{collapsing relation}, on $X$ by \[x \sim y  \mbox{ if } x \in \pi^{-1}(z) \mbox{ and } y \in \pi^{-1}(z) \] for some $z \in X$ where $\pi^{-1}(z)$ is a fiber of $S$. Since the subset $T \subset S$ is transverse to each fiber each element $t \in T$ appears in exactly one equivalence class.  For any $x \in S$ let $x' = T \cap [x]$. 

\begin{lemma}
The collapsing relation is an equivalence relation.
\end{lemma}

\begin{proof} Let $f = \pi^{-1}(w)$ be fiber of the S.  Reflexivity follows immediately because if $x \in f$ then $x \in f$ as well,  and therefore $x \sim x$ . The relation is symmetric because saying $x \in f$ and $y \in f$ is the same as saying $y \in f$ and $x \in f$.  Lastly transitivity holds because if $x \sim y$ and $y \sim z$ then $x \in f$ and $y \in f$ holds as well as $y \in f$ and $z \in f$.  Clearly then $x \in f$ and $z \in f$ which is precisely $x \sim z$.  \end{proof}

\begin{definition}[Collapsing Map]\label{collapsing map}
Let $S$ be a collapsing set in $X$.  The map $\phi: X \to \sfrac{X}{\sim}$ is given by $\phi(x) = [x]$ where $[x]$ is the equivalence class of $x$ under the relation $\sim$ on $S$.  We denote this map by the triple $(X,S,\phi)$ and 
write $\sfrac{X}{\sim} = X^*$.   
\end{definition}

To obtain a semi-metric on the collapsed space we deform the distance between points that are closer to the collapsing set than they are to each other.  We do this as follows.  Let $r_x$ denote the minimum distance from $x$ to the collapsing set $S$.  This is realized by taking balls of radius $r$ about $x$ and increasing the radius until the ball intersects $S$.  This distance is the least such radius.  If $d(x,y) > r_x + r_y$ we say the pair $\{x,y\}$ is in the \textit{vicinity} of $S$.  Note that if $x \in S$ then $r_x = 0$ which implies $d(x,y) > r_x + r_y$  for all $y \in X$.  That is, if $x \in S$ then $\{x,y\}$ is in the vicinity of $S$ for all $y \in X$.

\begin{definition}
Given a metric $\rho$ on $X$ and collapsing map $(X,S,\phi)$ we define $\rho_\phi : X^* \times X^* \rightarrow \mathbb{R}$ as follows: 
Let $[x],[y] \in X^*$. If $\{x,y\}$ is not in the vicinity of $S$ we define $$\rho_\phi([x],[y]) = \rho(x,y)$$ and 
 if $\{x,y\}$ is in the vicinity of $S$ we define $$\rho_\phi([x],[y]) = \rho_p(x',y') + r_x + r_y $$ where $x' = T \cap [x]$ and $y' = T \cap [y]$ and $\rho_p$ is the path metric in $T$ inherited from $\rho$.

\end{definition}

\begin{theorem} Let $(X,\rho)$ be a semi-metric space.  The map $\rho_\phi: X^* \times X^* \longrightarrow \mathbb{R}$ given by collapsing map $\phi$ gives a semi-metric space $(X^*, \rho_\phi)$.
\end{theorem}
\begin{proof}
Let $[x],[y] \in X^*$. If $\{x,y\}$ is not in the vicinity of $S$ then $\rho_\phi([x],[y]) = \phi(x,y)$.  In this case $\rho_\phi$ satisfies the properties of a semi-metric since $\rho$ does.  Now assume $\{x,y\}$ is in the vicinity of $S$.
\begin{enumerate}
    \item Since $\rho_p$ is a semi-metric and $r_x \geq 0$ for all $x\in X$ we have $$\rho_p(x',y') + r_x + r_y \geq 0$$ from which it is immediate that $\rho_\phi([x],[y]) \geq 0$.  
    \item If $\rho_\phi([x],[y]) = 0$ then $r_x = r_y = 0$ which implies  $x \in S$ and $y \in S$ and $[x] = x'$ and $[y]=y'$.  So $\rho_p(x',y') = 0$ which implies $x' = y'$ since $\rho_p$ is a semi-metric.  Therefore $[x] = [y]$. 
\end{enumerate}
\end{proof}

\section{Collapsing Map is a Quasi-Isometry}
We know that if we apply a quasi-isometry to a metric space we get a semi-metric space, so to shed light on the collapsing map we'd like to show that it is a quasi-isometry.  In Lemma \ref{ratio} we compare the length of points on $T$ to their length of their image under the collapsing map.  We show that the ratio of the new distance to the original distance is bounded above by a constant that depends on the Lipschitz constant of $T$.  

\begin{center}
\includegraphics[scale=.8]{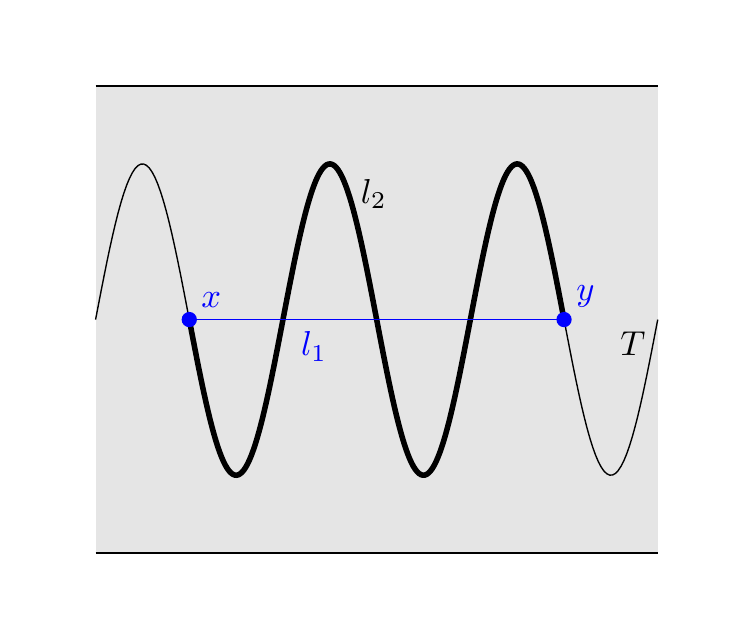}
\captionof{figure}{Two paths between $x$ and $y$.}
\end{center}

\begin{lemma}\label{ratio}
Let $T$ be a Lipschitz curve in the $(b,c)-$metric space $(X,\rho)$.  For all points $x, y$ in $T \subset X$ in  the ratio of their distance in the collapsed metric $\rho_\phi$ to that of the original metric $\rho$ is bounded.  That is, \[\forall \{x,y\} \in T, \exists K_L \in \mathbb{R} \mbox{ such that } \frac{\rho_\phi([x],[y])}{\rho(x,y)} \leq K_L .\]
\end{lemma}

\begin{proof}

 Since $T$ is a Lipschitz curve we know there is a map $f:(T,\rho) \longrightarrow (T,\rho_p)$ such that \[\rho_p(x,y) \leq L \cdot \rho(x,y)\]  for a fixed $L \in \mathbb{R}$ and any $x,y \in T$.  
 
 The collapsed metric essentially agrees with the path metric on $T$ we now show.  For $x,y \in T$ the distance of each point to $S$ is zero so
 \[\rho_\phi([x],[y]) = \rho_p(x',y') + r_x + r_y = \rho_p(x',y') \]
 and since $T$ is a Lipschitz curve in $X$ with Lipschitz constant $L$ we have 
  \[\rho_\phi([x],[y]) = \rho_p(x',y') \leq K \cdot \rho(x,y) \] 
  and the lemma follows.
 
  
  
\end{proof}

Note that since $x' \in T \subset \mathbb{R}^n$ it makes sense to evaluate the metric $\rho$ on $x'$ and $y'$.  

\begin{center}
\includegraphics[scale=.8]{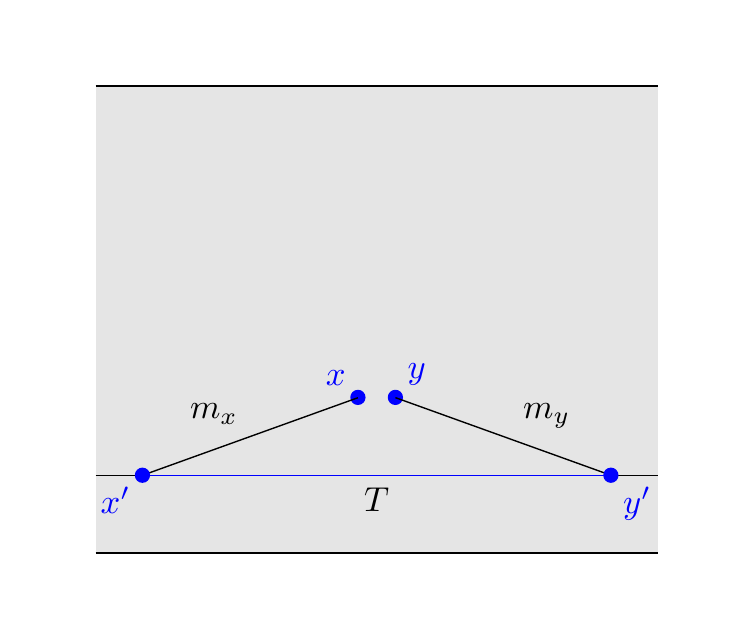}
\captionof{figure}{The lengths of $m_x$ and $m_y$ must each be less than the length of the longest fiber.}
\end{center}

The next lemmas will all be based on the following assumptions.  Let $(X,\rho)$ be a $(b,c)-$metric space and assume $\phi$ is a collapsing map defined on a collapsing set $S$.  Let $(X^*, \rho_\phi)$ be the semi-metric space from Theorem \ref{collapsing map}.  Because $S$ is a bounded fibered manifold, there exists a fiber of maximal length and this length we call $f$.  In the Lemma \ref{fiber} we compare distances between points in $S$ and their images under the collapsing map.  

\begin{lemma} \label{fiber}
Let $S$ be a collapsing set and let $x,y \in S$ and  $x',y' \in T$ as defined as in Definition \ref{collapsing set}. 
 The $\rho$-distance between $x'$ and $y'$ differs from the $\rho$-distance of $x$ and $y$ by an additive constant determined by the length of the longest fiber.  That is, 
    \[\rho(x',y') \leq b^2 \rho(x,y) + (b^2f + bc + c).\]
    where $f$ is the length of the longest fiber of $S$.
\end{lemma}
\begin{proof}
If $\rho_1(x,y) \geq \rho_1(x',y')$ we are done since the fibering is non-trivial.  First consider the case that $b=1$ and $c=0$, the case that $(X,\rho)$ is a metric space.  Note that the mapping moves a point $x \in S$ to another point on the same fiber.  Thus the maximum distance $x'$ could be from $x$ is $f$.  Therefore from the triangle inequality 
\[ \rho(x,y') \leq \rho(x,y) + \rho(y,y') \leq \rho(x,y) + f  \] and 
\[ \rho(x',y') \leq \rho(x',x) + \rho(x,y') \leq f + \rho(x,y) + f \]
we see \[\rho(x',y') \leq \rho(x,y) + 2f.\]

More generally, in a $(b,c)-$metric space
\[ \rho(x,y') \leq b(\rho(x,y) + \rho(y,y'))+c \leq b\rho(x,y) + bf + c \] and 
\[ \rho(x',y') \leq b(\rho(x',x) + \rho(x,y')) \leq b(f + \rho(x,y')) + c \leq bf + b^2 \rho(x,y) + b^2f + bc  \]
we see \[\rho(x',y') \leq b^2 \rho(x,y) + (b^2f + bc + c).\]

\end{proof}

We begin our comparison of the $\rho$ and $\rho_\phi$ metrics begin by constructing bounds $K$ and $C$ that satisfy the \textit{upper} half of the quasi-isometry inequalities.  Furthermore these bounds will apply in the case that points be farther apart than the length of the longest fiber (all in the $\rho-$metric). 

\begin{lemma} \label{upper farther}
 There are constants $K,C$ such that \[\rho_\phi([x],[y]) \leq K \rho(x,y) + C\] for all $x,y \in S$ such that $\rho(x,y) \geq f$ where $f$ is the length of the longest fiber of $S$. 
\end{lemma}

\begin{proof}
 Then $\forall$ $x, y \in S : \rho$(x,y) $\geq$ f, we wish to show that $\rho_\phi([x],[y]) \leq K \rho(x,y)$.      Let $x',y'$ be defined as in Definition \ref{collapsing set}.  Therefore both $\{x',y'\} \in T$.  Since the equivalence class is independent of representative Lemma \ref{ratio} gives
\[\rho_\phi([x],[y]) = \rho_\phi([x'],[y']) \leq K_L \rho(x',y').\] Using Lemma \ref{fiber} substitution gives 
\[ \rho_\phi([x],[y]) \leq b^2 K_L \rho(x,y) + K_L(b^2 f  + bc + c) .\]
Therefore
\[\rho_\phi([x],[x]) \leq K \rho(x,y) + C .\] where $K = b^2 K_L $ and $C = K_L(b^2 f  + bc + c)$.
\end{proof}

Continuing with the upper $K$ we now construct bounds in the case that points are closer together than the length of the longest fiber.

\begin{lemma} \label{upper nearer}
There are constants $K,C$ such that \[\rho_\phi([x],[y]) \leq K \rho(x,y) + C\] for all $x,y \in S$ where $\rho(x,y) < f$. 
\end{lemma}

\begin{proof}
  Assume $x,y \in S$ and $\rho(x,y) < f$.  Recall from Lemma \ref{fiber} that $\rho(x',y') \leq \rho(x,y) + 2f$.  Starting from $x'$ and $y'$ we wish to move in $T$ to find points $x'', y''$ that are exactly $\rho(x,y) + 2f$ apart.  Assuming $T$ is unbounded this is possible.  If $T$ is unbounded the theorem is trivial, so let $x'', y'' \in T$ and assume $\rho(x'',y'') = \rho(x,y)+2f$.  Therefore 
\[\rho_\phi([x],[y]) = \rho_\phi([x'],[y']) \leq \rho_\phi([x''],[y'']).\]
Since $\rho(x'',y'') > f$, we apply Lemma \ref{upper farther} to $x'',y''$ so we have 
\[\rho_\phi([x''],[y'']) \leq b^2 K_L \rho(x'',y'') + K_L(b^2 f  + bc + c) \]
and by construction 
\[\rho(x'',y'') \leq (\rho(x,y) + 2f) \] so
\[\rho_\phi([x''],[y'']) \leq b^2 K_L ((\rho(x,y) + 2f) ) + K_L(b^2 f  + bc + c)\]
or 
\[ \rho_\phi([x''],[y'']) \leq  (b^2 K_L) \rho(x,y) + K_L(3b^2 f + bc + c)\]

Thus 
\[\rho_\phi([x],[x]) \leq K \rho(x,y) + C .\] where $K = b^2 K_L $ and $C = K_L(3b^2 f + bc + c)$.
\end{proof}

Up to now for all points $x,y \in S$ we have found a $K$ and $C$ that gives an upper bound for a quasi-isometry.  We now extend this to all of $\mathbb{R}^n$. 

\begin{lemma} \label{upper}
There are constants $K,C$ such that \[\rho_\phi([x],[y]) \leq K \rho(x,y) + C\] for all $x,y \in X$.
\end{lemma}

\begin{proof}
Let $r_x$ denote the radius of the smallest ball (in the $\rho$ metric) centered at $x \in X$ that intersects $S$.  Let $x_S$ denote a point in $S$ such that $\rho(x,x_S) = r_x$. 
 
  For any $x \in S$ Lemmas \ref{upper farther} and \ref{upper nearer} imply that constants $K,C$ exist.  Now assume $x,y \notin S$.  By definition $\rho(x,y) \leq r_x + r_y$ then \[\rho_\phi([x],[y]) = \rho(x,y)\] so the metric is unchanged.  Now assume $\rho(x,y) > r_x + r_y$.  This implies that is \{x,y\} is in the vicinity of $S$ so again by definition, \[\rho_\phi([x],[y]) = r_x + r_y + d_T(x', y')\] where $x' = T \cap [x]$.
 
 The vicinity assumption implies
\begin{equation}
\label{d2d1da}
    \rho_\phi([x],[y]) \leq \rho(x,y)+d_T(x',y').
\end{equation}

Since $x',y' \in T$, Lemma \ref{ratio} implies
\begin{equation}
\label{da}
d_T(x',y') \leq K_L(\rho(x',y')).
\end{equation}
and since
Lemma \ref{fiber}  implies
 \[\rho(x',y') \leq b^2 \rho(x,y) + (b^2f + bc + c)\]
we have
\begin{equation}
 d_T(x',y') \leq K_L (b^2 \rho(x,y) + (b^2f + bc + c) ) .
\label{dTx'y'}
\end{equation}


Combining equations (\ref{d2d1da}) and (\ref{dTx'y'}) gives
\[\rho_\phi([x],[y]) \leq (b^2 K_L+1)\rho(x,y) + (b^2f + bc + c) ) \]


By taking $K = b^2 K_L+1$ and $C = (b^2f + bc + c)$ the Lemma holds.

\end{proof}

It is worth nothing that the inequality $\ref{dTx'y'}$ above is the upper bound to proving that the $\rho$-geodesic between $x$ and $y$ is quasi-isometric to an arc of $T$ bounded by $x'$ and $y'$.

\begin{lemma} \label{lower}
There are constants $K,C$ such that \[\frac{\rho(x,y)}{K} - C \leq \rho_\phi([x],[y]) \mbox{ } \forall x,y \in X\]
\end{lemma}

\begin{proof}
Let $x,y \in X$ with $x \neq y$ and assume $\rho(x,x_S) = r_x$ and $\rho(y,y_S) = r_y$.  By Lemma \ref{fiber} there exist constants $M,N$ for which. 
\begin{eqnarray}
   \rho(x_S,x') \leq Mf + N  \label{mfn1} \\
   \rho(y_S,y') \leq Mf+N.   \label{mfn2}
\end{eqnarray}
  
   Let $\rho(x,x') = l$, $\rho(y,y') = m$ and $\rho(x',y') = n$.   By applying the relaxed triangle inequality twice we see there are constants $b \geq 1,c \geq 0$ such that
\[\rho(x,y) \leq  b^2(\rho(x,x')+\rho(x',y') + \rho(y',y)) + b(c+1) = b^2(l+n+m) + b(c+1).\] and dropping $n$ gives
\begin{equation}
    \frac{\rho(x,y)}{b^2} \leq l+m + \frac{c+1}{b}.
    \label{l+m}
\end{equation}
By the relaxed triangle inequality and inequality (\ref{mfn1}) and (\ref{mfn2}) 
\begin{eqnarray}
 l \leq b(r_x+Mf + N) + c \label{l} \\ 
 m \leq b(r_y+Mf + N) + c. \label{m}
\end{eqnarray}
 
Combining inequalities (\ref{l+m}), (\ref{l}) and (\ref{m}) gives
\[\frac{\rho(x,y)}{b^2} \leq b(r_x + r_y + 2b(Mf + N)) + 2c + \frac{c+1}{b}\] so 
\begin{equation}
\frac{\rho(x,y)}{b^3} - 2(bMf + N) - \frac{c(2b+1)+1}{b^2} 
\label{amf}
\end{equation}
By definition $\rho_\phi([x],[y]) = r_x+r_y+d_T(x',y')$ and $x\neq y$ implies $d_T(x',y') = \epsilon > 0$ so 
\[\rho_\phi([x],[y]) = r_x+r_y + \epsilon \].  We set $K = b^3$ and $C = 2(bMf + N) - \frac{c(2b+1)+1}{b^2}$ and inequality (\ref{amf}) gives
\[\frac{\rho(x,y)}{K} - C = \rho_\phi([x],[y]) - \epsilon\]
and therefore
\[ \frac{\rho(x,y)}{K} - C \leq \rho_\phi([x],[y]) .\]
\end{proof}
 
Note that the lower bound $K$ depends only on $b$ coming from the $(b,c)-$ metric, and not on the Lipschitz constant.  In particular, for metric spaces the length of the longest fiber controls how much distances are decreased by the collapsing map.  


\begin{theorem}
    Let $(X,\rho)$ be a $(b,c)-$metric space and let $\phi: X \longrightarrow \sfrac{X}{\sim}$ be a collapsing map that collapses subset $S$ via equivalence relation $\sim$. If $\rho_\phi$ is the collapsed metric on $\sfrac{X}{\sim}$ then there are constants $K$ and $C$ such that 
    \[\frac{1}{K}\rho(x,y) - C \leq \rho_\phi([x],[y]) \leq  K\rho(x,y)+C \mbox{ } \forall x,y \in X \]
    Then $\phi$ a quasi-isometry between $(X,\rho)$ and $(\sfrac{X}{\sim},\rho_\phi)$.
   
\end{theorem}

\begin{proof}
Let $K = b^3 L + 1$ and Let \[C = K(3b^2 f + bc + c).\] Since $L \geq 1$ ($T$ is a Lipschitz manifold) and $n \geq 1$ and $b \geq 1$ these definitions of $K$ and $C$ satisfy both Lemma \ref{upper} and \ref{lower}.  
Therefore $\phi$ is a quasi-isometry between (X,$\rho$) and ($\sfrac{X}{\sim}$,$\rho_\phi$).
\end{proof}

\begin{exmp}
Let $S = \{(x,y) \in \mathbb{R}^2 | \sin(x) -1 < y < \sin(x)+1$  be the region between two sine curves and let $T = \{(x,\sin(x))\}$ for all $x\in \R$. The set $T$ is a (globally) Lipschitz (n-1)-manifold.  Collapsing $S$ to $T$ and defining a collapsed metric as above yields a quasi-isometry between $\mathbb{R}^2$ and the set $\sfrac{\mathbb{R}^2}{\sim}$.
\end{exmp}

\begin{corollary} \label{cor collapse}
If $X$ is a (b,c)-metric space with collapsing map $\phi: X \rightarrow Y$, then $Y$ is a $(b',c')$-metric space.
\end{corollary}

\begin{proof}
By Theorem 1 and Theorem 2, we see that $\phi$ is a quasi-isometry and that Y = $(X^*,\rho_\phi)$ is a semi-metric space. Since we have these conditions, Lemma 1.1 applies and we see Y is a $(b',c')$-metric space where $b'$ =$bK^2$ and $c' = 2bCK^2 + Kc + C.$
\end{proof}

Note: if c, C = 0 $\implies$ $c'$ = 0 makes this a Lipschitz map between b-metrics.

\section{Conclusion}

In most literature, quasi-isometries have the domain and range be metric spaces. This does not necessarily need to be the case as seen in \cite{Gray} (note the alternative definition for semi-metric). Assuming we have a metric we have shown that a quasi-isometry will preserve the (b,c)-triangle inequality by Lemma 1.1. The space of all quasi-isometries is vast and it seems to be that many canonical examples of quasi-isometries can be represented as a series of collapsing maps. In example \ref{example_1} we see that floor function is a quasi-isometry. This can also be realized as countably many collapsing maps. Expanding on this a natural conjecture is:

\begin{conjecture}
Any quasi-isometry can be expressed as the composition of a collapsing map and Lipschitz map.
\end{conjecture}

It is also possible to generalize both (b,c)-metrics as well as quasi-isometries by extending the bounds to functions instead of linear terms. The generalized quasi-isometry would not preserve a lot of structure, but it would preserve some depending on the function used. There is also analogous collapsing maps accompanying these new definitions. 

\begin{exmp}
Here we define the generalized quasi-isometry to be:
$$\frac{1}{Q}\rho(x,y)^2 + \frac{1}{K}\rho(x,y)-C \leq \rho_\phi(f(x),f(y)) \leq Q \rho(x,y)^2+  K \rho(x,y) + C$$
for some constants $C, Q \in [0,\infty), K \in [1,\infty)$.
The generalized (b,c)-metric to be:
 \begin{enumerate}
    \item $d(x,y) \geq 0$ with $d(x,y) = 0 \iff x = y$
    \item $d(x,y) = d(y,x) $
    \item $d(x,z) \leq a(d(x,y)^2 + d(y,z)^2) + b(d(x,y) + d(y,z)) + c$
 \end{enumerate}

And the generalized collapsing map to be not have the restriction to lengths of fibers as well as instead of being Lipshitz it is 2-H\"older Continuous. Since the generalized collapsing map is 2-H\"older Continuous, the generalized quasi-isometry holds and by generalizing lemma 1.1 it produces an (a,b,c)-metric.

\end{exmp}

This example leads us to our second conjecture:

\begin{conjecture}
If we have a generalized quasi-isometry where our bounds are given by a function, a generalized (b,c)-metric and generalized collapsing map whose bounds are given my the same function then we have the generalized form of Theorem 2 and Corollary 2.1.
\end{conjecture}

\printbibliography

@article{Tom,
    author = {Suzuki Tomonari},
    title = {Basic inequality on a b-metric space and its applications},
    journal = {J Inequal Appl},
    year = {2017},
    volume = {256},
    number = {1},
    DOI = {https://doi.org/10.1186/s13660-017-1528-3}
}

@book{Dan,
author = {Benson Farb and Dan Margalit},
doi = {doi:10.1515/9781400839049},
url = {https://doi.org/10.1515/9781400839049},
title = {A Primer on Mapping Class Groups},
year = {2011},
publisher = {Princeton University Press},
ISBN = {9781400839049}
}

@book{Clay,
    author ={Matt Clay and Dan Margalit},
    title = {Office Hours with a Geometric Group Theorist},
    year = {2017},
    publisher = {Princeton University Press}
}

@article{Luu,
title = {Lipschitz and quasiconformal approximation of homeomorphism pairs},
journal = {Topology and its Applications},
volume = {109},
number = {1},
pages = {1-40},
year = {2001},
issn = {0166-8641},
doi = {https://doi.org/10.1016/S0166-8641(99)00145-5},
url = {https://www.sciencedirect.com/science/article/pii/S0166864199001455},
author = {Jouni Luukkainen}
}

@article{Gray,
 ISSN = {00029947},
 URL = {http://www.jstor.org/stable/23513361},
 author = {Robert Gray and Mark Kambites},
 journal = {Transactions of the American Mathematical Society},
 number = {2},
 pages = {555--578},
 publisher = {American Mathematical Society},
 title = {Groups acting on semimetric spaces and quasi-isometries of monoids},
 volume = {365},
 year = {2013}
}

\end{document}